\documentclass[12pt, a4paper]{article}
\usepackage[T1]{fontenc}
\usepackage{latexsym}
\usepackage{amssymb}
\usepackage{amsmath}
\usepackage{amsthm}
\usepackage{mathrsfs}
\usepackage{ucs}

\newtheorem{sat}{{\sc Theorem}}[section]
\newtheorem{lem}[sat]{{\sc Lemma} }
\newtheorem{kor}[sat]{Corollary} 
\newtheorem{defi}[sat]{Definition}   
\newcounter{saveeqn}

\title{H-Regular Borel measures on locally compact abelian  groups}
\date{}
\begin{document}
\author{{\bf {\footnotesize BY}}\\ 
{\bf {\normalsize L. Klotz}}~{\sc (Leipzig)}\\
{\footnotesize AND} {\bf {\normalsize J.M. Medina}}~{\sc (Buenos Aires)}}
\maketitle
\noindent
{\it Abstract.} {\small 
Let $G$ be an LCA group, $H$ a closed subgroup, $\varGamma$ the dual group of $G$. In accordance with analogous notions in prediction theory 
the classes of $H$-regular and $H$-singular Borel measures on $\Gamma$ are defined. A characterization of $H$-regular measures is given and
a Wold type decomposition is obtained. 
}\\[0.2cm]
\maketitle

\setcounter{section}{0} 
\section{Introduction}
\noindent
Let $G$ be an LCA group, i.~e. a locally compact abelian group with Hausdorff topology, whose group operation is written additively. Let $H$ be a closed subgroup, $\pi_H $ the canonical homeomorphism from $G$ onto the factor group $G/ H=: \mathop{G}\limits^{\sim} $ and $\mathop{x}\limits^{\sim} := \pi_{H} (x)$ the equivalence class of $x\in G$.  
Denote  by $\varGamma$  the dual group of $G$, by $\langle{\gamma,x}\rangle$ the value of $\gamma\in \varGamma$ at $x \in G$, by  $ \varLambda:=\{\gamma \in \varGamma: \langle{\gamma,y}\rangle=1$ for all $y\in H \}$ the annihilator group of $H$, and the Borel subsets of $\varGamma$ by $\mathcal{B}(\varGamma)$ . Recall that the factor group $\mathop{G}\limits^{\sim}$ can be identified with the dual group of $\varLambda$ by setting $\langle{\lambda,\mathop{x}   \limits^{\sim}}\rangle: =\langle{\lambda, x}\rangle$, $\lambda\in\Lambda,\,\mathop{x}\limits^{\sim}\in \mathop{G}\limits^{\sim}$, where $x$ is an arbitrary element of $\pi_H ^{-1}(\mathop{x}\limits^{\sim})$.
  Throughout the present paper we shall assume that  the set  $\varLambda$ is finite or countably infinite.
If $S$ is a non-empty subset of $G$, a trigonometric $S$-polynomial or, more precisely, a trigonometric polynomial with frequencies from $S$ is a function $p:\Gamma \longrightarrow \mathbb{C}$, which is a finite sum of the form:
$$p(\cdot) = \sum\limits_{k} a_k \langle{\cdot ,s_k}\rangle\, $$
$a_k \in\mathbb{C}$, $s_k \in S$. By $\mathbf{P}(S)$ denote the linear space of all trigonometric $S$-polynomials. If $x_1$ and $x_2$ belong 
to the same $H$-coset, then $\mathbf{P}(x_1 +H)=\mathbf{P}(x_2 +H)$, and we set $\mathbf{P}(x+H)=:\mathbf{P}(\mathop{x}\limits^{\sim})$, $x\in G$.

Let $\mu$ be a regular finite  non-negative  measure on the Borel $\sigma$-algebra $\mathcal{B}(\varGamma)$ (a measure on $\mathcal{B}(\varGamma)$ in short). Here the notion of regularity means that for all $B\in\mathcal{B}(\varGamma)$ and all $\epsilon >0$ there exists a compact set $K$ and an open subset $U$ of $\varGamma$ such that $K\subset B \subset U$ and $\mu (U\setminus K)<\epsilon$ \cite{Cohn80}.  If $\alpha \in (0,\infty)$, let $L^\alpha(\mu)$ be the metric space
of ($\mu$-equivalence classes of) Borel measurable $\mathbb{C}$-valued functions on $\varGamma$, which are $\alpha$-integrable with respect to $\mu$. 
The closure of a subset $\mathcal{L} \subset L^{\alpha} (\mu)$ with respect to the metric of $L^{\alpha} (\mu)$ is denoted by $C_{\alpha}\mathcal{L}$.
For $\alpha \in (1,2]$, 
the space $L^\alpha(\mu)$ can be interpreted  as the spectral domain of a harmonizable symmetric $\alpha$-stable process on $G$ \cite{Weron85}, particularly, 
$L^2(\mu)$ is the spectral domain of a certain stationary Gaussian process. Motivated by  the corresponding definitions in prediction theory of such processes we introduce the following notions.
\begin{defi}\label{d1}
A measure $\mu$ on $\mathcal{B}(\varGamma)$ is $H$-{\bf{regular}} if and only if:
$$\mathop{\bigcap}\limits_{x\in G} C_{\alpha}\mathbf{P}(x+H)= \mathop{\bigcap}\limits_{\mathop{x}\limits^{\sim}\in \mathop{G}\limits^{\sim}} C_{\alpha} \mathbf{P}(\mathop{x}\limits^{\sim})=\{0\}\,.$$
It is called $H$-{\bf{singular}} if $C_{\alpha}\mathbf{P} (\mathop{x}\limits^{\sim})=L^{\alpha} (\mu)$ for all $\mathop{x}\limits^{\sim}\in \mathop{G}\limits^{\sim}$.
\end{defi}
$H$-singularity is closely related to Whittaker-Shannon-Kotelnikov sampling problems. In the case $\alpha=2$ similar notions are studied in \cite{Lloyd59} and \cite{Lee78}  in the context of sampling of wide sense stationary random processes in $\mathbb{R}$ and certain classes of finite variance random processes indexed over LCA groups, respectively. These sampling problems are, in some sense, equivalent to find some completeness conditions for certain systems of trigonometric polynomials in $L^2 (\mu)$, where $\mu$ can be regarded as the spectral measure associated to certain classes of finite variance random processes. These  results are generally stated as conditions on the translates of the support of  the measure  $\mu$. Other examples of this principle can be seen in \cite{Kluva65} and \cite{Bea07}. A more general result in this direction which will be useful later is the following:  
 
\begin{sat}\label{t1}(\cite{Klotz16}, Theorem 4.5)
A measure $\mu$ on $\mathcal{B}(\varGamma)$ is $H$-singular if and only if there exists a set $B\in\mathcal{B}(\varGamma)$ such that
$\mu(\varGamma \setminus B)=0$ and $B\bigcap (\lambda +B)=\varnothing$ for all $\lambda \in \Lambda \setminus \{0\}$.
\end{sat}
The preceding theorem shows that the notion of $H$-singularity does not depend on $\alpha$. We shall prove that the notion of $H$-regularity is
independent of $\alpha$ as well. Therefore the formal dependence of definition \ref{d1} on the parameter $\alpha$ can be ignored.
Generalizing a method of \cite{Klotz05}, which was applied to  $G=\mathbb{Z}$ and $H=n\mathbb{Z}$, $n\in\mathbb{N},$ we shall describe the set of all $H$-regular measures in section \ref{sec2} of this note. On the other hand, Wold type decompositions is a classic topic  in approximation and prediction  theory \cite{Gren57} and in abstract Hilbert space theory up to the present e.g. \cite{Ezza15}. In this context this prompts to obtain a Wold type decomposition of an arbitrary measure  on
$\mathcal{B}(\varGamma)$, see theorem \ref{t3} below. Finally, theorem \ref{t4} is devoted to the characterization of those measures, for which the spaces
$\mathbf{P}(\mathop{x}   \limits^{\sim})$, are pairwise orthogonal in $L^2 (\mu)$.
We mention that the assertion of theorem \ref{t1} remains true if the condition that $\Lambda$ is countable is replaced by alternative assumptions on
$\Lambda$, $\Gamma$ or $\mu$, \cite{Klotz16}. It would be of interest to obtain a characterization of $H$-regular measures under similar conditions.
\section{Characterization of $H$-regular measures}\label{sec2}
A subset $T$ of $\varGamma$ is called a transversal (with respect to $\Lambda$) if it meets each
$\Lambda$-coset just once. Note that $T$ is a transversal if and only if the set $\pi^{-1} _{\Lambda}(\pi_{\Lambda} (T))=\mathop{\bigcup}\limits_{\lambda\in\Lambda} (T +\lambda)$ is equal to $\Gamma$ and $T \cap (T+\lambda) =\varnothing$ for all $\lambda \in \Lambda\setminus\{0\}$, (\cite{Klotz16}, Lemma 3.3).
Since $\Lambda$ is discrete and, hence, metrizable, by theorem 1 of \cite{Feld68} there exists a transversal, which belongs to $\mathcal{B}(\varGamma)$.
Let $\mu$ be a measure on $\mathcal{B}(\varGamma)$ and $T\in \mathcal{B}(\varGamma)$ be a transversal, for $\lambda \in \Lambda$, define a measure
$\mu_{\lambda}$ by $\mu_{\lambda}(B):=\mu (B\cap (T+\lambda))$, a measure $\nu_{\lambda}(B)=\mu_{\lambda}(\lambda + B)$, for $B\in\mathcal{B}(\varGamma)$, and set $\nu:= \sum\limits_{\lambda\in\Lambda} \nu_{\lambda}$. All measures just defined are, indeed, regular finite non-negative measures on $\mathcal{B}(\varGamma)$ (\cite{Klotz16}, Lemma 2.1). Note that $\nu(\varGamma)=\nu(T)$ .\hfill\\
If $\mathop{x}   \limits^{\sim} \in \mathop{G}   \limits^{\sim}$  and $\varphi$ is a $\mathbb{C}$-valued function on $\varGamma$, define a function
$ V_{\mathop{x}   \limits^{\sim}} \varphi$ by $(V_{\mathop{x}   \limits^{\sim}} \varphi)(\gamma):= \langle{\lambda, \mathop{x}   \limits^{\sim}}\rangle \varphi (\gamma -\lambda)$, $\gamma\in \lambda + T$ , $\lambda\in\Lambda$. If $S\subseteq \varGamma$, let $\mathbf{1}_{S}$ stand for the indicator function of $S$.
\begin{lem}\label{l1} Let $\alpha\in (0,\infty)$, for any $\mathop{x}   \limits^{\sim}\in \mathop{G}   \limits^{\sim}$, the map $V_{\mathop{x}   \limits^{\sim}}$ establishes an isometric isomorphism between $L^{\alpha}(\nu)$ and the subspace $C_{\alpha} \mathbf{P}(\mathop{x}   \limits^{\sim}) $ of $L^{\alpha} (\mu)$ satisfying:
\begin{equation}\label{e1}
V_{\mathop{x}   \limits^{\sim}} p =p,
\end{equation}  
for all $p \in \mathbf{P}(\mathop{x}   \limits^{\sim})$, and
\begin{equation}\label{e2}
V^{-1} _{\mathop{x}   \limits^{\sim}} f =f\mathbf{1}_T \in L^{\alpha}(\nu),
\end{equation}  
for all $f\in C_{\alpha} \mathbf{P}(\mathop{x}   \limits^{\sim})$.
\end{lem} 
\begin{proof}
Let $\varphi\in L^{\alpha}(\nu)$ . It is clear that $V_{\mathop{x}   \limits^{\sim}} \varphi$ is $(\mathcal{B}(\varGamma),\mathcal{B}(\mathbb{C}))$-measurable. Moreover,
$$\int\limits_{\varGamma} |(V_{\mathop{x}   \limits^{\sim}} \varphi)(\gamma)|^{\alpha} d\mu(\gamma) = \sum\limits_{\lambda\in\Lambda} \int\limits_{\lambda + T} |(V_{\mathop{x}   \limits^{\sim}} \varphi)(\gamma)|^{\alpha} d\mu(\gamma)=$$
$$  \sum\limits_{\lambda\in\Lambda} \int\limits_{\lambda + T} |\langle{\lambda, \mathop{x}   \limits^{\sim}}\rangle \varphi (\gamma -\lambda)|^{\alpha} d\mu(\gamma) = \sum\limits_{\lambda\in\Lambda} \int\limits_{ T} | \varphi(\gamma)|^{\alpha} d\nu_{\lambda}(\gamma)=$$
$$\int\limits_{ T} | \varphi(\gamma)|^{\alpha} d\nu(\gamma)= \int\limits_{ \varGamma} | \varphi(\gamma)|^{\alpha} d\nu (\gamma)\,,$$
which shows that $V_{\mathop{x}   \limits^{\sim}}$ is an isometry from $L^{\alpha}(\nu)$ into $L^{\alpha}(\mu)$. If $p\in\mathbf{P}(\mathop{x}   \limits^{\sim})$,
$p(\cdot)=\sum\limits_k a_k \langle{\cdot,x+y_k}\rangle$, where $x\in \pi_{\Lambda} ^{-1}(\{\mathop{x}   \limits^{\sim}\})$, $a_k\in\mathbb{C}$, $y_k \in H$, then for for $\gamma\in \lambda + T $, one has 
$$(V_{ \mathop{x}   \limits^{\sim}}p)(\gamma)=\sum\limits_k a_k \langle{\lambda,x}\rangle \langle{\gamma-\lambda, x+y_k }\rangle=\sum\limits_k a_k \langle{\gamma, x+y_k }\rangle =p(\gamma),\,\lambda\in\Lambda,$$
which yields eq. \ref{e1} since $\varGamma=\mathop{\bigcup}\limits_{\lambda\in\Lambda} (\lambda +T)$. Consequently $C_{\alpha} \mathbf{P}(\mathop{x}   \limits^{\sim})\subseteq V_{\mathop{x}   \limits^{\sim}}\left ({L^{\alpha}(\mu)}\right) $. The opposite inclusion is true as well since from theorem \ref{t1} it follows that $ \mathbf{P}(\mathop{x}   \limits^{\sim})$ is dense in $L^{\alpha}(\nu)$. Thus, the range of $V_{ \mathop{x}   \limits^{\sim}}$  is equal to $C_{\alpha} \mathbf{P}(\mathop{x}   \limits^{\sim})$.\hfill\\
If $(p_j)_{j\in\mathbb{N}}$ is a sequence of $\mathbf{P}(\mathop{x}   \limits^{\sim})$ tending to $f$ in $L^{\alpha}(\mu)$, then
$(V^{-1} _{\mathop{x}   \limits^{\sim}} p_j )_{j\in\mathbb{N}}$ tends to $V^{-1} _{\mathop{x}   \limits^{\sim}} f$ in $L^{\alpha} (\nu)$.
Choosing an appropriate subsequence, we can suppose that
\begin{equation}\label{e3}
\mathop{\rm{lim}}\limits_{j\longrightarrow\infty} p_j = f\;\;\mu-a.e.
\end{equation}
and
\begin{equation}\label{e4}
\mathop{\rm{lim}}\limits_{j\longrightarrow\infty} V^{-1} _{\mathop{x}   \limits^{\sim}} p_j = V^{-1} _{\mathop{x}   \limits^{\sim}} f\;\;\nu-a.e.
\end{equation}
Relation \ref{e3} and the definition of $\nu$ imply that
\begin{equation}\label{e5}
\mathop{\rm{lim}}\limits_{j\longrightarrow\infty}  p_j \mathbf{1}_{T}= f\mathbf{1}_T \;\;\nu-a.e.
\end{equation}
Taking into account eq. \ref{e1}, we get eq. \ref{e2} by eqs. \ref{e4},\ref{e5}.
\end{proof}
For $\lambda\in\Lambda$, let $h_{\lambda}$ be the Radon-Nikodym derivative of $\nu_{\lambda}$ with respect to $\nu$. We can assume that
$h_{\lambda}$ is a $(\mathcal{B}(\varGamma),\mathcal{B}([0,\infty)))$-measurable function and $h_{\lambda} =0$ on $\varGamma\setminus T$. Note that
\begin{equation}\label{e6}
\sum\limits_{\lambda\in\Lambda} h_{\lambda} =1 \;\;\;\nu-a.e.
\end{equation}
\begin{sat}\label{t2} A measure $\mu$ on $\mathcal{B}(\varGamma)$ is $H$-regular if and only if 
\begin{equation}\label{e7}h_{\lambda}<1\;\;\;\nu-a.e. 
\end{equation}
for all
$\lambda\in\Lambda$.

\end{sat}
\begin{proof}
If eq. \ref{e7} is not satisfied, there exists $\kappa \in\Lambda$ and $B\in\mathcal{B}(\varGamma)$ such that $\nu(B)>0$ and
$h_{\lambda} =1$ on $B$. If $ \mathop{x}   \limits^{\sim} \in\mathop{G}   \limits^{\sim}$, the function $\varphi:= \langle{-\kappa, \mathop{x}   \limits^{\sim}}\rangle\mathbf{1}_{B}$ is a non-zero element of $L^{\alpha}(\nu)$ and 
$$V_{\mathop{x}   \limits^{\sim}} \varphi =\sum\limits_{\lambda\in\Lambda} \langle{\lambda -\kappa, \mathop{x}   \limits^{\sim}}\rangle\mathbf{1}_{\lambda+B}\,,$$
by definition of $V_{\mathop{x}   \limits^{\sim}}$. If $\lambda\in \Lambda\setminus \{\kappa\}$, eq. \ref{e6}
yields $h_{\lambda}=0$ $\nu$-a.e. on $B$, which implies that $\mu (\lambda +B)=\mu_{\lambda}(\lambda +B)=0$, hence $V_{\mathop{x}   \limits^{\sim}} \varphi =\mathbf{1}_{\kappa +B}$ in $L^{\alpha}(\mu)$. By lemma \ref{l1} the function $\mathbf{1}_{\kappa +B}$ is a non-zero element of
$C_{\alpha}\mathbf{P}(\mathop{x}   \limits^{\sim})$, $\mathop{x}   \limits^{\sim} \in \mathop{G}   \limits^{\sim}$, a contradiction to $H$-regularity. Now assume that eq. \ref{e7} is satisfied and $f\in \mathop{\bigcap}\limits_{\mathop{x}   \limits^{\sim}\in \mathop{G}   \limits^{\sim}} C_{\alpha}\mathbf{P}(\mathop{x}   \limits^{\sim}) $ . Let $\mathop{x}   \limits^{\sim} \in \mathop{G}   \limits^{\sim}$, from the definition of
$V_{\mathop{x}   \limits^{\sim}}$ and eq. \ref{e4} it follows that for all $\lambda,\kappa \in\Lambda$ one has  
$$\langle{\lambda , -\mathop{x}   \limits^{\sim}}\rangle f(\lambda + \gamma)=f(\gamma)=\langle{\kappa ,-\mathop{x}   \limits^{\sim}}\rangle f(\kappa + \gamma)\,,$$
for $\nu$-almost all $\gamma \in T$. Since $\Lambda$ is countable, there exists a set $B_{\mathop{x}   \limits^{\sim}} \in\mathcal{B}(\varGamma)$, such that $B_{\mathop{x}   \limits^{\sim}} \subset T$, $\nu (\varGamma \setminus B_{\mathop{x}   \limits^{\sim}})=0$, and
\begin{equation}\label{e8}\langle{\lambda , -\mathop{x}   \limits^{\sim}}\rangle f(\lambda + \gamma)=f(\gamma)=\langle{\kappa , -\mathop{x}   \limits^{\sim}}\rangle f(\kappa + \gamma)
\end{equation}
for all $\lambda,\kappa\in{\Lambda}$, and all $\gamma \in B_{\mathop{x}   \limits^{\sim}}$. Let 
$B_{\lambda}:=\{\gamma\in T:\; h_{\lambda}(\gamma)\neq 0\}$, $B_{\lambda\,\kappa}:=B_{\lambda} \cap B_{\kappa}$, $\lambda,\kappa \in \Lambda$, $\lambda\neq \kappa$, $B=\mathop{\bigcup}\limits_{\lambda,\,\kappa\in \Lambda; \lambda\neq \kappa} B_{\lambda\,\kappa}$, and note that eq. \ref{e7} is equivalent to
the condition 
\begin{equation}\label{e9}
\nu(\varGamma\setminus B)=0.
\end{equation}
 For $\lambda,\kappa \in\Lambda$, $\lambda\neq \kappa$, choose $\mathop{x}   \limits^{\sim} \in \mathop{G}   \limits^{\sim}$ such that 
 $\langle{\lambda , -\mathop{x}   \limits^{\sim}}\rangle \neq \langle{\kappa , -\mathop{x}   \limits^{\sim}}\rangle$ and define 
 $B'_{\lambda\,\kappa} := B_{\lambda\,\kappa} \cap B_{\mathop{0}   \limits^{\sim}} \cap B_{\mathop{x}   \limits^{\sim}}$.
 If $\gamma\in B'_{\lambda\,\kappa}$, from eq. \ref{e8} one obtains the following homogeneous linear system of equation with respect to the
 unknowns $f(\lambda +\gamma)$ and $f(\kappa +\gamma)$:
 $$\begin{cases} f(\lambda +\gamma) - f(\kappa+\gamma)=0 \\
\langle{\lambda , -\mathop{x}   \limits^{\sim}}\rangle f(\lambda +\gamma) - \langle{\kappa, -\mathop{x}   \limits^{\sim}}\rangle f(\kappa+\gamma)=0 \end{cases}\,.$$ 
Since the coefficient matrix of this system is invertible, it follows $f(\lambda +\gamma)=f(\kappa +\gamma)=f(\gamma)=0$. Since $\nu(B'_{\lambda\,\kappa})=\nu(B_{\lambda\,\kappa})$ and the set $\Lambda$ is countable, we can conclude from eq. \ref{e9} that $f\mathbf{1}_T =0$ $\nu$-a.e. hence $f=0$
in $L^{\alpha} (\mu)$ by eq. \ref{e2}.  
 \end{proof}
From the description of $H$-regular and $H$-singular measures respectively one can easily derive a Wold type decomposition of any measure on $\mathcal{B}(\varGamma)$.
\begin{sat}\label{t3}
Any measure on $\mathcal{B}(\varGamma)$ can be decomposed into a sum of an $H$-regular measure $\mu_{\rho}$ and an $H$-singular measure 
$\mu_{\sigma}$. For $\alpha\in (0,\infty)$, the spaces $L^{\alpha}(\mu_{\rho})$ and $L^{\alpha}(\mu_{\sigma})$ can be identified with subspaces of
$L^{\alpha}(\mu)$ of the form $\mathbf{1}_{B_{\rho}}L^{\alpha}(\mu)$ and  $\mathbf{1}_{B_{\sigma}}L^{\alpha}(\mu)$, respectively, where $B_{\rho},B_{\sigma} \in \mathcal{B}(\varGamma)$. Then $L^{\alpha}(\mu)=L^{\alpha}(\mu_{\rho}) \oplus L^{\alpha}(\mu_{\sigma})$ and $L^{\alpha}(\mu_{\sigma})= \mathop{\bigcap}\limits_{\mathop{x}   \limits^{\sim}\in \mathop{G}   \limits^{\sim}} C_{\alpha} \mathbf{P}(\mathop{x}   \limits^{\sim})$.
\end{sat}
\begin{proof}
Let $T\in\mathcal{B}(\varGamma)$ be a transversal and $h_{\lambda}$ the Radon Nikodym derivative of the corresponding measure $\nu_{\lambda}$ , $\lambda\in\Lambda$ , with respect to $\nu=\sum\limits_{\lambda\in\Lambda} \nu_{\lambda}$. Define 
$$B'_{\rho} :=\left\{{\gamma\in T:\; h_{\lambda}(\gamma)<1\;for\,all\,\lambda\in\Lambda}\right\}\,,$$
$$B'_{\sigma} :=\left\{{\gamma\in T:\; h_{\lambda}(\gamma)=1\;for\,some\,\lambda\in\Lambda}\right\}\,,$$
$B_{\rho}:=\pi^{-1} _{\Lambda} (\pi_{\Lambda} (B'_{\rho})) ,\,B_{\sigma}:=\pi^{-1} _{\Lambda} (\pi_{\Lambda} (B'_{\sigma}))$, and measures $\mu_{\rho}$, $\mu_{\sigma}$ by:
$$\mu_{\rho} (B):= \mu(B\cap B_{\rho})\,,\;\mu_{\sigma} (B):= \mu(B\cap B_{\sigma})\,,\;B\in\mathcal{B}(\varGamma) \,.$$
Generalizing the arguments of (\cite{Klotz05}, pp. 296-297) in a straightforward way, one can show that the measures $\mu_{\rho}$ and $\mu_{\sigma}$ have all the asserted properties. 
\end{proof}
Since the space $L^2 (\mu)$ is a Hilbert space, it arises the problem of characterizing those measures $\mu$, for which the linear spaces $C_2 \mathbf{P}(\mathop{x}   \limits^{\sim})$, $\mathop{x}   \limits^{\sim}\in \mathop{G}   \limits^{\sim}$ are pairwise orthogonal. 
\begin{sat}\label{t4}
Let $\mu$ be a measure  on $\mathcal{B}(\varGamma)$. Let $T\in\mathcal{B} (\varGamma)$ be a transversal and $h_{\lambda}$ the Radon-Nikodym derivative of the corresponding measure $\nu_{\lambda}$ with respect to $\nu=\sum\limits_{\lambda\in\Lambda} \nu_{\lambda}$. If $\Lambda$ consists of $n$ elements, $n\in\mathbb{N}$, then the spaces $C_2 \mathbf{P}(\mathop{x}   \limits^{\sim})$, $\mathop{x}   \limits^{\sim}\in \mathop{G}   \limits^{\sim}$, are pairwise orthogonal in $L^2 (\mu)$ if and only if $h_{\lambda}=\frac{1}{n}$, $\nu$-a.e., for every $\lambda\in\Lambda$.  If $\Lambda$ is countably infinite, then 
$C_2 \mathbf{P}(\mathop{x}   \limits^{\sim})$, $\mathop{x}   \limits^{\sim}\in \mathop{G}   \limits^{\sim}$, do not constitute a family of pairwise orthogonal subspaces of $L^2 (\mu)$.
\end{sat}
\begin{proof}

It is easy to see that for $\mathop{x}   \limits^{\sim}\in \mathop{G}   \limits^{\sim}$, the map $(U_{\mathop{x}   \limits^{\sim}} f)(\gamma):= \langle{\lambda, -\mathop{x}   \limits^{\sim}}\rangle f(\gamma)$, $\gamma \in \lambda +T$, $\lambda \in \Lambda$, is an isometric isomorphism of $L^2 (\mu)$. Therefore the linear spaces $C_2 \mathbf{P}(\mathop{x}   \limits^{\sim})$, $\mathop{x}   \limits^{\sim}\in \mathop{G}   \limits^{\sim}$, are pairwise orthogonal if and only if the space  $ \mathbf{P}(\mathop{0}   \limits^{\sim})= \mathbf{P}(H)$ is orthogonal to all spaces $\mathbf{P}(\mathop{x}   \limits^{\sim})$, $\mathop{x}   \limits^{\sim}\in \mathop{G}   \limits^{\sim}\setminus \{\mathop{0}   \limits^{\sim}\}$.. By lemma \ref{l1}, the spaces  $ \mathbf{P}(\mathop{0}   \limits^{\sim})$ and $\mathbf{P}(\mathop{x}   \limits^{\sim})$ are orthogonal if and only if:
$$0=\int\limits_{\varGamma} (V_{\mathop{0}   \limits^{\sim}} \varphi) (\gamma) \overline{(V_{\mathop{x}   \limits^{\sim}} \psi) (\gamma)} d\mu (\gamma)= \sum\limits_{\lambda\in\Lambda} \int\limits_{\lambda +T} \varphi (\gamma-\lambda) \overline{ \langle{\lambda,\mathop{x}   \limits^{\sim} }\rangle\psi(\gamma-\lambda)} d\mu_{\lambda} (\gamma)$$
$$=\sum\limits_{\lambda\in\Lambda} \int\limits_{T} \langle{\lambda,-\mathop{x}   \limits^{\sim} }\rangle \varphi (\gamma) \overline{ \psi(\gamma)} d\nu_{\lambda} (\gamma)= \int\limits_{T}  \varphi (\gamma) \overline{ \psi(\gamma)}\left({\sum\limits_{\lambda\in\Lambda}\langle{\lambda,-\mathop{x}   \limits^{\sim} }\rangle h_{\lambda} (\gamma)}\right) d\nu (\gamma)\,,$$
$\varphi,\psi \in L^2 (\nu)$, which is equivalent to the existence of a set $B_{\mathop{x}   \limits^{\sim}} \in \mathcal{B}(\varGamma)$ such that
$B_{\mathop{x} \limits^{\sim} } \subseteq T$, $\nu(\varGamma\setminus B_{\mathop{x} \limits^{\sim} })=0$ and
$$\sum\limits_{\lambda\in\Lambda}\langle{\lambda,-\mathop{x}   \limits^{\sim} }\rangle h_{\lambda} (\gamma) =0,\;\;\textrm{for all}\,\gamma\in B_{\mathop{x} \limits^{\sim} }\,.$$
If $\Lambda$ is finite, its dual group $\mathop{G}   \limits^{\sim}$ is finite as well. Setting $\mathop{\bigcap}\limits_{\mathop{x}   \limits^{\sim} \in \mathop{G}   \limits^{\sim}\setminus \{\mathop{0}   \limits^{\sim}\}} B_{\mathop{x}   \limits^{\sim}}$, we have $B\in\mathcal{B}(\varGamma)$, $B\subseteq T$, $\nu(\varGamma\setminus B)=0$, and for all $\gamma\in B$ and all $\mathop{x} \limits^{\sim} \in \mathop{G} \limits^{\sim}\setminus \{\mathop{0} \limits^{\sim}\} $,
\begin{equation}\label{e10}
\sum\limits_{\lambda\in\Lambda}\langle{\lambda,-\mathop{x}   \limits^{\sim} }\rangle h_{\lambda} (\gamma) =0\,.
\end{equation}
If $\Lambda$ contains exactly $n$ elements, $n\in\mathbb{N}$, then its dual group $\mathop{G}   \limits^{\sim}$ contains exactly $n$ elements and 
hence $\mathop{G}   \limits^{\sim}$ can be identified, for some $s\in\mathbb{N}$ with a discrete group $\mathbb{Z}(m_1)\times\dots\times\mathbb{Z}(m_s)$, where $m_j \in \mathbb{N}$, $\mathop{\Pi}\limits_{j=1}^s m_j =n$, and $\mathbb{Z}(m_j)$ denotes the group of integers $\{0,\dots,m_j -1\}$ with the group operation of addition modulo $m_j$. The dual group  $\Lambda$ of $\mathop{G}   \limits^{\sim}=\mathbb{Z}(m_1)\times\dots\times\mathbb{Z}(m_s)$ can be identified with the set of functions $\lambda$ of the form: 
$$\lambda((k_1,\dots,k_s))=\prod\limits_{j=1}^s exp\left({\frac{2\pi i k_j l_j}{m_j}}\right)\,,(k_1,\dots, k_s)\in \mathbb{Z}(m_1)\times\dots\times\mathbb{Z}(m_s)\,,$$
 where $(l_1,\dots, l_s)\in \mathbb{Z}(m_1)\times\dots\times\mathbb{Z}(m_s)$. It follows that for $\gamma\in B$, the system of eq. \ref{e10} can be written as a system of $\prod\limits_{j=1}^s m_j -1$ linear equations
 \begin{equation}\label{e11}
\sum\limits_{l_1=0}^{m_1 -1} \dots \sum\limits_{l_s =0}^{m_1 -1} \prod\limits_{j=1}^s exp\left({\frac{2\pi i k_j l_j}{m_j}}\right) h_{(l_1,\dots,l_s) }(\gamma)=0,
 \end{equation} 
$\,(k_1,\dots, k_s)\in \mathbb{Z}(m_1)\times\dots\times\mathbb{Z}(m_s)\setminus\{(0,\dots,0)\} \,,$
with respect to the  $\mathop{\Pi}\limits_{j=1}^s m_j =n$ unknown quantities $h_{(l_1,\dots,l_s) }(\gamma)$,$\,(l_1,\dots, l_s)\in \mathbb{Z}(m_1)\times\dots\times\mathbb{Z}(m_s)$. By eq. \ref{e6} we can assume that $\sum\limits_{\lambda\in\Lambda} h_{\lambda} (\gamma) =1$. Adding this equation to eqs. \ref{e11} as the first equation, we obtain a linear system, whose coefficient matrix is invertible since it is the tensor product of the $s$ invertible Vandermonde matrices   $\left({ exp\left({\frac{2\pi r t}{m_j}}\right)}\right)_{r,t=0,\dots,m_j -1}$, $j\in\{1,\dots,s\}$, thus the system has a unique solution.
If $k_j \neq 0$ for some $j\in\{1,\dots,s\}$, then $\sum\limits_{l_j=0}^{m_j-1} exp\left({\frac{2\pi i k_j l_j}{m_j}}\right)=0$, hence
$$\sum\limits_{l_1=0}^{m_1 -1} \dots \sum\limits_{l_s =0}^{m_1 -1} \prod\limits_{j=1}^s exp\left({\frac{2\pi i k_j l_j}{m_j}}\right) =0\,,$$
which implies that  $h_{(l_1,\dots,l_s) }(\gamma)=\frac{1}{n}$ for all $\,(l_1,\dots, l_s)\in \mathbb{Z}(m_1)\times\dots\times\mathbb{Z}(m_s)$. Using the initial notation: $h_{\lambda} (\gamma)=\frac{1}{n}$ for all $\lambda\in\Lambda$ is a solution.\hfill\\
If $\Lambda$ is countably infinite, it is discrete and $\sigma$-compact. Therefore its dual group is compact and metrizable , cf. (\cite{Hewitt63}, Theorem 23.17) and (\cite{Morris77}, Theorem 29), respectively. It follows that $\mathop{G}   \limits^{\sim}$ is separable. Let $\mathop{D}   \limits^{\sim}$ be a countable dense subset of $\mathop{G}   \limits^{\sim}$. Since $\mathop{G}   \limits^{\sim}$ is not discrete, we can require that
$\mathop{0}   \limits^{\sim} \notin \mathop{D}   \limits^{\sim}$. The set $A:=\mathop{\bigcap}\limits_{\mathop{x}   \limits^{\sim}\in \mathop{D}   \limits^{\sim}} B_{\mathop{x}   \limits^{\sim}} $ belong to $\mathcal{B}(\varGamma)$, $A\subseteq T$, $\nu(\varGamma \setminus A)=0$ and from eq. \ref{e10} it follows that 
$$\sum\limits_{\lambda\in\Lambda}\langle{\lambda,-\mathop{x}   \limits^{\sim} }\rangle h_{\lambda} (\gamma) =0$$ 
for all $\gamma\in A$ and $\mathop{x}   \limits^{\sim} \in \mathop{D}   \limits^{\sim}$. For $\gamma\in A$, the left-hand side of the preceding equality is the value of  the Fourier transform of the function $\lambda \longmapsto h_{\lambda} (\gamma)$, $\lambda\in \Lambda$, at the point $\mathop{x}   \limits^{\sim} \in \mathop{D}   \limits^{\sim}$. Since $\mathop{D}   \limits^{\sim}$ is dense in $\mathop{G}   \limits^{\sim}$, it follows  that the Fourier transform of the function $\lambda \longmapsto h_{\lambda} (\gamma)$ is identically 0, hence, the function itself is $0$ a.e. with respect to the Haar measure on $\Lambda$. Since $\Lambda$ is discrete, we get $h_{\lambda} (\gamma)=0$ for all $\lambda\in\Lambda$ and all $\gamma\in A$, which implies that
$\nu$ and, hence, $\mu$ is the zero measure.
	\end{proof}
	The condition $h_{\lambda}=\frac{1}{n}$ $\nu$-a.e. , $\lambda\in\Lambda$, means that $\mu(B)=\mu(\lambda+B)$ for all $\lambda\in\Lambda$ and
	$B\in\mathcal{B}(\varGamma)$, $B\subseteq T$, i.e., the measure $\mu$ is periodic with respect to $\Lambda$ and $T$. Thus the following corollary of theorem \ref{t4} can be stated.
	\begin{kor}\label{c1} Let $\Lambda$ be a finite subgroup of an LCA group $\varGamma$ and $\mu$ a measure on $\mathcal{B}(\varGamma)$.
	If there exists a transversal $T$ such that $\mu$ is periodic with respect to $\lambda$ and $T$, then $\mu$ is periodic with respect to $\Lambda$ and any transversal.  
	\end{kor} 
	\subsection{Characterization of the projection onto $C_2\mathbf{P}(\mathop{x}\limits^{\sim})$}
	Considering the Hilbert space $L^2(\mu)$, motivated by some interpolation problems, if $\mu$ is not $H$-singular, it is of practical interest to find explicit expressions for $P_{\mathop{x}\limits^{\sim}}$, the orthogonal projection onto $C_2\mathbf{P}(\mathop{x}\limits^{\sim})$, $\mathop{x}\limits^{\sim}\in \mathop{G}\limits^{\sim}$. A classical example when $G=\mathbb{R}$, $H=\mathbb{Z}$ can de found in cf. (\cite{Gren57}, sec.2.4) and the same notion appears in \cite{Lloyd59}, where a particular case of theorem \ref{t1} over $\mathbb{R}$ is studied and related to Whittaker-Shannon-Kotelnikov sampling expansions for wide sense stationary processes.\hfill\\
	For $\lambda \in \Lambda$ define a measure $\rho_{\lambda}(B):=\mu(\lambda +B)$, for $B\in\mathcal{B}(\varGamma)$ and set $\rho:=\sum\limits_{\lambda\in\Lambda} \rho_{\lambda}$. Note that $\rho_{0}=\mu$, for all $\lambda\in\Lambda$, $\;\rho(\cdot)=\rho(\cdot +\lambda)$ and $\rho_{\lambda}$ is absolutely continuous with respect to $\rho$, and that $\rho$ is $\sigma$-finite since $ \varLambda$ is assumed to be countable.
	 These measures are related by a technical lemma whose proof is immediate.
	\begin{lem}\label{l2} Let $\mu$ be a measure on $\mathcal{B}(\varGamma)$, $T\in\mathcal{B}(\varGamma) $ a transversal,  $g$  the Radon-Nikodym derivative of $\mu$ with respect to $\rho$, $f\in L^1(\mu)$ and $\kappa\in\Lambda$. Then:
	$$\int\limits_{\varGamma} f(\gamma)d\mu(\gamma)= \int\limits_{T+\kappa} \sum\limits_{\lambda\in\Lambda} f(\gamma+\lambda) g(\gamma +\lambda) d\rho(\gamma)\,.$$
	\end{lem}
	
  \subsubsection{Remark} \label{rem1}
 If $p\in\mathbf{P}(\mathop{x}\limits^{\sim})$ and $\lambda\in\Lambda$, then $$p(\gamma +\lambda)=\sum\limits_{k} a_k \langle{\gamma +\lambda ,x+y_k}\rangle= \langle{\lambda,x}\rangle \sum\limits_{k} a_k \langle{\gamma ,x+ y_k}\rangle = \langle{\lambda,x}\rangle p(\gamma)\,,$$
 where $y_k \in H$. Therefore one can check that if $f\in L^{\alpha}(\mu)$ is such that $f(\cdot +\lambda)=\langle{\lambda\,,x}\rangle f(\cdot)$ $\mu$-a.e., then $f\in C_{\alpha} \mathbf{P}(\mathop{x}\limits^{\sim})$.
 A similar argument shows that $C_{\alpha} \mathbf{P}(\mathop{x}\limits^{\sim})=\langle{\cdot ,x}\rangle C_{\alpha} \mathbf{P}(\mathop{0}\limits^{\sim})$.
 \hfill\\
 The following theorem  gives an expression for the orthogonal projection of a function $f$ onto $C_2\mathbf{P}(\mathop{x}\limits^{\sim})$, $\mathop{x}\limits^{\sim}\in \mathop{G}\limits^{\sim}$. Its proof goes in a similar  vein to cf. (\cite{Sal69}, theorem 3.2) so some details are omitted. 
 \begin{sat}\label{t5} Let  $\mu$ be a measure on $\mathcal{B}(\varGamma)$, $T\in\mathcal{B}(\varGamma)$ a transversal and $g$ the Radon-Nikodym derivative of $\mu$ with respect to $\rho$.  Then for $\mathop{x}\limits^{\sim}\in \mathop{G}\limits^{\sim}$ and $f\in L^2 (\mu)$, the orthogonal projection $P_{\mathop{x}\limits^{\sim}}f$ of $f$ onto $C_2\mathbf{P}(\mathop{x}\limits^{\sim})$ is given by
 \begin{equation}\label{proj} 
 P_{\mathop{x}\limits^{\sim}} f= \sum\limits_{\lambda\in\Lambda} {\langle{-\lambda,x}\rangle} f(\cdot +\lambda) g(\cdot +\lambda)\;\;\rho-a.e\;(and\; then\,\mu-a.e. )
 \end{equation} 
 \end{sat}
 \begin{proof}
 Taking into account that $P_{\mathop{x}\limits^{\sim}} f(\cdot)= {\langle{\lambda,x}\rangle} (P_{\mathop{x}\limits^{\sim}} f) (\cdot -\lambda)$ $\mu$-a.e. and $\varphi (\cdot)= {\langle{\lambda,x}\rangle} \varphi(\cdot -\lambda) $ $\mu$-a.e., for $\lambda\in\Lambda$, $f\in L^2 (\mu)$, $\varphi \in C_2\mathbf{P}(\mathop{x}\limits^{\sim})$, for a fixed $\kappa\in\Lambda$, write the orthogonality condition on $P_{\mathop{x}\limits^{\sim}} f$:
 $$0=\int\limits_{\varGamma} ((P_{\mathop{x}\limits^{\sim}} f)(\gamma)-f(\gamma))\overline{\varphi (\gamma)} d\mu(\gamma)$$
 $$=\sum\limits_{\lambda\in\Lambda}\int\limits_{\kappa +T} ( P_{\mathop{x}\limits^{\sim}} f (\gamma) - \langle{-\lambda,x}\rangle f(\gamma+\lambda)) \overline{\varphi (\gamma)} g(\gamma +\lambda) d\rho (\gamma)\,.$$
 Since $\sum\limits_{\lambda\in\Lambda} g(\cdot +\lambda)=1$ $\rho$-a.e.,
the last equality  implies that
\begin{equation}\label{e12}\int\limits_{\kappa +T} (P_{\mathop{x}\limits^{\sim}} f)(\gamma)\overline{\varphi (\gamma)}  d\rho (\gamma)= \int\limits_{\kappa +T} \left({\sum\limits_{\lambda\in\Lambda}\langle{-\lambda,x}\rangle f(\gamma +\lambda) g(\gamma +\lambda) }\right)\overline{\varphi (\gamma)}  d\rho (\gamma)\,.
\end{equation}
The last interchange of the sum with the integral is justified since by lemma \ref{l2}:
\begin{equation}
\int\limits_{\kappa +T} {\sum\limits_{\lambda\in\Lambda}|P_{\mathop{x}\limits^{\sim}} f(\gamma)-\langle{-\lambda,x}\rangle f(\gamma +\lambda)|  }|{\varphi (\gamma)}| g(\gamma +\lambda) d\rho (\gamma)
\end{equation}
$$\leq\int\limits_{\varGamma} (|P_{\mathop{x}\limits^{\sim}} f(\gamma)|+|f(\gamma )|)  |{\varphi (\gamma)}| d\mu(\gamma) \leq 2 \left\|f\right\|_{L^2 (\mu)} \left\|\varphi\right\|_{L^2 (\mu)}<\infty\,.$$
In particular eq. \ref{e12} holds  if $\varphi\in C_2\mathbf{P}(\mathop{x}\limits^{\sim})$ in eq.\ref{e12} is taken as $$\varphi (\cdot)=\langle{\cdot ,x}\rangle \sum\limits_{\lambda\in\Lambda} \mathbf{1}_B (\cdot +\lambda)\,,$$ for any $B\subseteq T$, $B\in\mathcal{B}(\varGamma)$ and therefore 
eq. \ref{proj} holds $\rho$-a.e. on $T+\kappa$, for $\kappa\in\Lambda$. The desired result follows from this since $\varGamma =\mathop{\bigcup}\limits_{\lambda\in \Lambda} T+\lambda$.

 \end{proof}
 \begin{kor}
 Under the same hypothesis of theorem \ref{t5}, if $\alpha \in [1,\infty]$, eq. \ref{proj} defines a bounded projection from $L^{\alpha}(\mu)$ onto $C_{\alpha}\mathbf{P}(\mathop{x}\limits^{\sim})$.
\end{kor}
\begin{proof}
Given $f\in L^2(\mu) \subset L^1 (\mu)$, $|(P_{\mathop{x}\limits^{\sim}}f)(	\gamma)|\leq \sum\limits_{\lambda\in\Lambda}|f(\gamma +\lambda) g(\gamma +\lambda)| $, thus by lemma \ref{l2},
$$\int\limits_{\varGamma} |P_{\mathop{x}\limits^{\sim}} f(\gamma)| d\mu(\gamma) \leq \int\limits_{T} \sum\limits_{\kappa\in\Lambda} \left({ \sum\limits_{\lambda\in\Lambda} |f(\gamma +\lambda +\kappa)| g(\gamma +\lambda+\kappa)}\right)g(\gamma +\kappa) d\rho(\gamma)\,, $$
but since $\sum\limits_{\kappa\in\Lambda} g(\cdot +\kappa)=1$ $\rho$-a.e. then the last equation equals:
$$\int\limits_{T}  { \sum\limits_{\lambda\in\Lambda} |f(\gamma +\lambda)| g(\gamma +\lambda)} d\rho(\gamma)= \int\limits_{\varGamma} |f(\gamma)| d\mu(\gamma)\,.$$
Therefore, taking into account theorem \ref{t5}, if $\alpha=1,\,2$, then  $$\left\|P_{\mathop{x}\limits^{\sim}} f\right\|_{L^{\alpha}(\mu)} \leq \left\| f \right\|_{L^{\alpha}(\mu)},\; \mathop{\rm{for}}\; f\in L^2(\mu).$$
The boundedness of $P_{\mathop{x}\limits^{\sim}}$ acting on $L^{\alpha}(\mu)$, $\,\alpha\in(1,2)$, is obtained by the Marcinkiewicz interpolation theorem and a duality argument proves the same for $\alpha\in (2,\infty)$.  From eq. \ref{proj}, if $f\in L^{\alpha}(\mu)$ and $\kappa\in \Lambda$, then it is easy to verify that $(P_{\mathop{x}\limits^{\sim}} f)(\gamma +\kappa)= \langle{\kappa,x}\rangle (P_{\mathop{x}\limits^{\sim}} f)(\gamma)$ and then $P_{\mathop{x}\limits^{\sim}} f \in C_{\alpha}\mathbf{P}(\mathop{x}\limits^{\sim})$ (See remark \ref{rem1}), therefore $P_{\mathop{x}\limits^{\sim}}(L^{\alpha}(\mu))\subseteq C_{\alpha}\mathbf{P}(\mathop{x}\limits^{\sim})$. The opposite inclusion follows since if $p\in \mathbf{P}(\mathop{x}\limits^{\sim})$, then it is straightforward to see that
$(P_{\mathop{x}\limits^{\sim}} p)(\gamma)=p(\gamma)$. 
\section*{Aknowledgements}
J.M. Medina thanks CONICET and Universidad de Buenos Aires, grant No. UBACyT 20020100100503.

\end{proof}

\noindent L. Klotz \\ 
Fakultät für Mathematik und Informatik\\
Universität Leipzig\\
04109 Leipzig, Germany
\hfill\\
\hfill\\
  J. M. Medina\\
Inst. Argentino  de Matem\'atica ``A. P. Calder\'on''- CONICET and Universidad de Buenos Aires-Departamento de Matem\'{a}tica.\\
Saavedra 15, 3er piso (1083),
Buenos Aires, Argentina\\
\end{document}